\documentclass[12pt]{amsart}

\makeatletter
\def\@settitle{\begin{center}%
		\baselineskip14\p@\relax
		\normalfont\Large%<- NEW
		\@title
	\end{center}%
}
\makeatletter
\def\@settitle{\begin{center}%
		\baselineskip14\p@\relax
		\normalfont\Large%<- NEW
		\@title
	\end{center}%
}
\makeatother

\usepackage{graphicx,epsfig}

\usepackage{bbm}
\usepackage[title]{appendix}
\usepackage{esvect}
\usepackage{mathtools}
\usepackage{tikz-cd}
\usepackage{fancyhdr}
\usepackage{amsfonts}
\usepackage{amsmath}
\usepackage{amssymb}
\usepackage{mathabx}
\usepackage{amsthm}
\usepackage{tikz-cd}
\usepackage{parskip}
\usepackage{times} %make sure that the times new roman is used

\usepackage{mathrsfs}
\usepackage{scalerel,amssymb}

\usepackage{amsmath,amssymb,latexsym, amsfonts, amscd, amsthm, xy}
\input{xy}
\xyoption{all}

\newcommand{\Z}{\mathbb{Z}}
\newcommand{\Q}{\mathbb{Q}}
\newcommand{\R}{\mathbb{R}}
\newcommand{\C}{\mathbb{C}}

\newcommand{\SL}{\mathrm{SL}}

\newcommand{\GL}{\mathrm{GL}}

\input{xy}
\xyoption{all}

\makeindex 

=650pt \headheight = 8pt \headsep = 10pt

\usepackage{fancyhdr}
\usepackage{amsfonts}
\usepackage{amsmath}
\usepackage{amssymb}
\usepackage{amsthm}
\usepackage{tikz}
\usepackage{tikz-cd}
\usepackage{parskip}
\usepackage{times} %make sure that the times new roman is used
\usepackage[
bibstyle=alphabetic,
citestyle=alphabetic,
hyperref=true,
backref=false]{biblatex}
\usepackage{hyperref}
\hypersetup{pdftoolbar=true, pdftitle={GKform}, pdffitwindow=true, colorlinks=true, citecolor=green, filecolor=black, linkcolor=blue, urlcolor=blue, hypertexnames=false}
\usepackage{cleveref}
\makeindex 

=650pt \headheight = 8pt \headsep = 10pt

\newtheorem{thm}{Theorem}[section]
\newtheorem{prop}[thm]{Proposition}
\newtheorem{lem}[thm]{Lemma}
\newtheorem{cor}[thm]{Corollary}

\theoremstyle{definition}
\newtheorem{rem}[thm]{Remark}

\numberwithin{equation}{section}

\title[Explicit formula for theta lift via Bruhat decomposition]{Explicit formula for the $(\GL_2,\GL_2)$ theta lift via Bruhat decomposition \large Formule explicite pour la correspondance thêta $(\GL_2,\GL_2)$ via la décomposition de Bruhat}
\author{Peter Xu}

\usepackage{setspace}
\begin{document}
	
	\makeatletter
	\@setabstract
	\makeatother
	\maketitle
	
	\begin{abstract}
		Using combinations of weight-$1$ and weight-$2$ of Kronecker-Eisenstein series to construct currents in the distributional de Rham complex of a squared elliptic curve, we find a simple explicit formula for the type II $(\GL_2,\GL_2)$ theta lift without smoothing, analogous to the classical formula of Siegel for periods of Eisenstein series. For $K$ a CM field, the same technique applies without change to obtain an analogous formula for the $(\GL_2(K),K^\times)$ theta correspondence. 
		
		\,
		
		Utilisant des combinaisons des séries et courants Kronecker-Eisenstein de poids $1$ et $2$ dans des complexes de Rham d'une courbe elliptique carrée, nous trouvons une formule simple explicite pour la correspondance thêta de type II pour $(\GL_2,\GL_2)$ sans stabilisation; cela est un analogue de la formule classique de Siegel pour les périodes des séries d'Eisenstein. Pour $K$ un corps avec CM, la même technique s'applique pareil pour obtenir une telle formule pour la correspondance $(\GL_2(K),K^\times)$.
	\end{abstract}
	
	\section{Introduction}
	
	In the article \cite{BCG} and its sequel \cite{BCGV}, the authors construct a theta lift from the $(n-1)$-homology of $\GL_n$ to modular forms for $\GL_2$. This construction uses an equivariant polylogarithm class; in particular, when $n=2$, the class in question corresponds to the algebraic $\GL_2$-action on the squred universal elliptic curve over the upper half-plane (or its modular curve quotients).
	
	This class can be computed by various means: in \cite{BCG}, the lifts of certain modular geodesics are computed via the seesaw principle, while in \cite{BCGV}, certain stabilizations of the lift are computed by relating them to values ``at the boundary,'' i.e. by partial modular symbols. Related cocycles valued in $K$-theory were constructed in \cite{SV} and \cite{BPPS} (for general $n$; the regulators of these cocycles these correspond to parabolic stabilizations of the theta lift. We also construct both parabolic and non-parabolic stabilized versions of the $K$-theory cocycle for general $n$ \cite{X3}.
	
	In the previous paper \cite{X4}, we constructed explicit elements in the $\GL_2$-equivariant distributional de Rham complex of a torus to re-derive classical ``unstabilized'' formulas for periods of weight-$2$ Eisenstein series. The principle used was that in cohomology, these Eisenstein series arise as pullbacks of an equivariant polylogarithm class on $(S^1)^2$, allowing us to compute them via equivariant-geometric means. In the present article, by applying the same technique to the squared universal elliptic curve, we obtain novel, similarly explicit formulas for the $(\GL_2,\GL_2)$-theta lift. These formulas especially shed light on the Eisenstein component of the lift, exhibiting the particular combiniation of weight-$2$ Eisenstein series which arise from specializations at parabolic matrices.
	
	Due to the nature of the theta lift's construction, instead of the equivariant distributional de Rham complex of \cite{X4}, we will use the equivariant weight-$2$ distributional \emph{Dolbeault} complex of our squared universal elliptic curve. The explicit elements in loc. cit. were comprised of weight-$1$ and weight-$2$ Bernoulli polynomials $B_1(z)$ and $B_2(z)$; these will be replaced with theta series $E_1(\tau,z)$ and $E_2(\tau,z)$ interpolating weight-$1$ and $2$ Eisenstein series. To find the explicit forms of these elements, as in loc. cit., we will need to use an explicit version of the Bruhat decomposition of 
	\[
	\GL_2(\Q)=BwB
	\]
	where 
	\[
	w=\begin{pmatrix} 0 & 1 \\ 1 & 0 \end{pmatrix}
	\]
	represents a Weyl element, and $B\subset \GL_2(\Q)$ is the upper triangular Borel subgroup.
	
	\subsection{Related work and future directions}
	
	A similar computation using complexes to compoute equivariant cohomology was completed in \cite[\S3]{KS}, in a more general setting using a generalized Kronecker-Eisenstein series. Our calculation can be considered as a more explicated form of the calculation there, for the self-product of the universal elliptic curve with trivial coefficients, except that:
	\begin{enumerate}
		\item we treat the entire action of $\GL_2$ at once, instead of just working with the equivariance under one subtorus at a time;
		\item we construct also a cocycle for a larger group with no geometric action on the fibers ($\GL_2(\Q)$ instead of $\GL_2(\Z)$, or $\GL_2(K)$ in the CM case);
		\item we do not rely on any stabilization or smoothing coming from a degree-zero torsion cycle.
	\end{enumerate}
	We restrict to this setting both because it is particularly arithmetically rich, and allows for a considerably simpler and less technical exposition. Futhermore, by eschewing smoothing, the resulting formulas are more transparently related to classical objects such as classical holomorphic Eisenstein series or the period formulas of Siegel (as in \cite{X4}), and corresponds directly to the theta kernel of integration of \cite{BCG}. Furthermore, our formulas can be specialized (with cohomological meaning) at arbitrary non-zero torsion points, thanks to Theorem \ref{thm:inj}.
	
	As we noted in our previous article \cite{X4} in a simpler setting, our same methods should produce analogous formulas with twisted coefficients corresponding to higher-weight automorphic sheaves; the resulting formulas will involve progressively more complicated combinations of higher-weight specializations of Kronecker-Eisenstein series. Another natural generalization would be to use the more complicated decomposition of parabolic subgroups for $\GL_n$ and $n$-fold products of elliptic curves; in this setting, our unsmoothed formulas grow in complexity very quickly with $n$, and it is not clear to us if there exists a symbol formalism that can help one understand this complexity. Furthermore, the direct analogue of Theorem \ref{thm:inj} does not hold for $n>2$, so the ability to specialize at torsion sections is more limited. Nevertheless, in the \emph{smoothed} setting, things are much simpler and can be captured by symbols coming from elementary linear algebra; we follow this approach for elliptic schemes in \cite{X3}.
	
	\subsection{Acknowledgements} 
	
	I would like to thank first and foremost Nicolas Bergeron, to whose ideas I owe all of my work in this subject area; and who very kindly accommodated me in Paris in February 2023, listening to and helping to develop my ideas even with the limited energy he had at the time. This article is dedicated in his memory. I would also like to thank Romyar Sharifi, Marti Roset Julia, and Pierre Charollois for helpful consultations during the writing process.
	
	\section{Eisenstein theta lift for $(\GL_2,\GL_2)$}
	
	The theta lift we are considering was originally constructed purely analytically in \cite{BCG}, but we will find the algebraic construction of \cite{BCGV} more useful; the comparison between the two is also proven in loc. cit. (See also the author's thesis \cite{X}.)
	
	To begin, we consider the complex upper half-plane $\mathcal{H}$, and the universal elliptic curve over it given by the complex uniformization
	\[
	E:=(\mathcal{H} \times \mathbb{C})/\{(\tau, \Z + \tau \Z)\}.
	\]
	Here, we will write $\tau=x+yi$ for the coordinate on the base, and $z$ for the fiberwise coordinate on $\C$. This fiber bundle has the complex analytic action (as a fiber bundle) of $\GL_2(\Z)$ by 
	\[
	\begin{pmatrix}
		a&b \\ c& d
	\end{pmatrix} (\tau, z) = \left(\frac{a\tau + b}{c\tau + d}, \frac{z}{c\tau + d}\right).
	\]
	We will also write $z=u-\tau v$ for $u,v\in \R$; this then allows $\GL_2(\Z)$ to act on the column vector $(u,v)$ by the standard left action.
	
	The principal space we will work on is the self-product $T:= E\times_\mathcal{H} E$, which we will coordinatize by $\tau$ on the base, and $z_1$ and $z_2$ on the two elliptic curves. For us, the important property of $T$ is that it has a fiberwise action by endomorphisms of $M_2(\Z)$, coming from the endomorphism action of $\Z$ on any elliptic scheme by isogenies, i.e. the integer $a\in \Z$ acts by the finite multiplication-by-$a$ isogeny $[a]_*$ of degree $a^2$. In particular, the group $\GL_2(\Z)$ acts by automorphisms on $T$.
	
	For any torsion-free subgroup $H\subset \text{GL}_2(\mathbb{Z})$, the $Y(H):= H \backslash \mathcal{H}$ can be equipped with an algebraic structure making it the fine moduli space of elliptic curves over $\C$-schemes with $H$-level structure. There is also the corresponding universal elliptic curve
	\[
	E(H) := H \backslash (\mathcal{H}\times (\C/\mathbb{Z}^2)).
	\]
	and its square
	\[
	T(H):= E(H)\times_{Y(H)} E(H),
	\]
	whose fiber over a point of the moduli space is the square of the corresponding elliptic curve. In this article, for technical reasons, we will primarily work over $\mathcal{H}$ rather than after quotient by $H$; however, the output of our construction will visibly descend through these quotients.
	
	Let $\Gamma\subset \GL_2(\Z)$ be any subgroup; we will work with the $\Gamma$-equivariant cohomology of $\Gamma$-subspaces of $T$. In this article, we will follow the approach pioneered by \cite{KS} to define the theta lift algebraically in coherent cohomology of relative sheaves of differentials over $\mathcal{H}$; see the author's thesis \cite{X} for details about how this relates to the algebraic construction of \cite{BCGV}.
	
	Following \cite{KS}, if we fix an auxiliary integer $c>1$, there is a unique (``polylogarithm'') cohomology class
	\[
	z_\Gamma^{(c)} \in H^{1}_\Gamma(T-T[c], \Omega^2_{T-T[c]/\mathcal{H}})^{(0)}
	\]
	characterized (as indicated by the superscript $(0)$) by being invariant by $[a]_*$ for all integers $a$ relatively prime to $c$, and having residue 
	\[
	[T[c]-c^2\{0\}]\in H^2_{\Gamma,T[c]}(T, \Omega^2_{T/\mathcal{H}})
	\]
	where $H^\bullet_{\Gamma,T[c]}$ denotes the equivariant cohomology with support. Let $S\subset T$ be a nonempty $\Gamma$-fixed arrangement of $T[c]$-translates of elliptic subgroups of $T$; we restrict the class $z_\Gamma^{(c)}$ to
	\[
	H^{1}_\Gamma(T-S, \Omega^2_{T-S/\mathcal{H}})^{(0)} := \varinjlim_{S_f\subset S} H^{1}_\Gamma(T-S_f, \Omega^2_{T-S_{f}/\mathcal{H}})^{(0)}
	\]
	where the limit is over finite subarrangements. The various $T-S_f$ are cofinally fiberwise Stein manifolds over the contractible base $\mathcal{H}$, so have vanishing higher coherent cohomology; thus, in the Hochschild-Serre spectral sequence we have a Hochschild-Serre edge map
	\[
	e:  H^{1}_\Gamma(T-S, \Omega^2_{T-S/\mathcal{H}})^{(0)} \to  H^1(\Gamma, H^0(T-S, \Omega^2_{T-S/\mathcal{H}})) := H^1(\Gamma, \varinjlim H^0(T-S_f, \Omega^2_{T-S_f/\mathcal{H}}))^{(0)} 
	\]
	where here we use the exactness of filtered colimits.
	
	Thanks to a similar argument as the one to uniquely define $z_\Gamma^{(c)}$, using projectors built from isogenies, $e(z_\Gamma^{(c)})$ can be refined to a class which we denote
	\[
	Z^{(c)}_\Gamma \in H^1(\Gamma, H^0(T-S,\Omega^2_{T-S/\mathcal{H}})^{(0)} )
	\]
	i.e. a cohomology class valued in prime-to-$c$ isogeny-fixed forms (instead of being isogeny-fixed only up to coboundary). 
	
	This group cohomology class valued in $2$-forms on (an open subspace of) the bundle $T$ parameterizes a family of theta lifts: if one contracts with the $\GL_2(\Z)$-fixed vector field $\partial z_1 \otimes \partial z_2$ and pulls back by any $\Gamma$-fixed torsion section $x$ disjoint from $S$, then 
	\[
	\Theta^{(c)}_{\Gamma, x}:= x^*(Z^{(c)}_\Gamma \cap \partial z_1 \otimes \partial z_2) \in H^1(\Gamma, H^0(\mathcal{H}, \omega^{\otimes 2}))
	\]
	can be viewed via integration as a map from $H_1(\Gamma)$ to a weight-$2$ modular form inside $H^0(\mathcal{H}, \omega^{\otimes 2})$, as the values of the integral will be fixed by any level structure $H$ stabilizing the section $x$. Note that for any given nonzero $x$, one can always find a disjoint $S\subset T$ invariant by its stabilizer, so all torsion sections have associated pullbacks. Furthermore, since one can always restrict away from the union of two different $S$, it is clear the pulled back class is independent of the choice of removed sub-elliptic curves.
	If we fix any point of the base $\tau\in \mathcal{H}$, we can similarly define the fiberwise version of these cycles by restriction,
	\[
	Z^{(c)}_\tau \in H^1(\Gamma, H^0(T_\tau-S_\tau,\Omega^2_{T_\tau-S_\tau})^{(0)} ),\Theta^{(c)}_{\tau, x}:= x^*(Z^{(c)}_\tau \cap \partial z_1 \otimes \partial z_2) \in H^1(\Gamma, \C)
	\]
	where the subscript $\tau$ refers to the fiber over that point. (We here omit the $\Gamma$ from the notation for brevity, since these classes are compatible under restriction of $\Gamma$ in any case.)
	
	Out of technical convenience, we will obtain our explicit formulas for the fiberwise classes $\Theta^{(c)}_{\tau, x}$. Since the action of $\Gamma$ is trivial for these classes, however, this will also imply formulas for the classes $\Theta^{(c)}_{\Gamma, x}$.\footnote{With a little extra work, one can also prove similar formulas for the ``spread out'' classes $Z^{(c)}_\Gamma$, but this is not the focus of this article.}
	
	\subsection{Computing with complexes}
	
	We now define the main tool we need, in this setting: the weight-$2$ \emph{distributional Dolbeault complex}. Let $X$ be a complex manifold of complex dimension $d$, and write $\mathcal{A}^{p,q}_X$ for the sheaves of smooth complex differentials on $X$ of holomorphic degree $p$ and antiholomorphic degree $q$; these sheaves are flabby on $X$, so we (a bit abusively) use this same notation for their global sections. These forms support the pair of anti-commuting Dolbeault differentials
	\[
	\partial: \mathcal{A}^{p,q}_{X} \to \mathcal{A}^{p+1,q}_{X}, \overline{\partial}: \mathcal{A}^{p,q}_{X}\to \mathcal{A}^{p,q+1}_{X}
	\]
	which are square zero, satisfy the Leibniz rule, and sum to the usual exterior derivative on differential forms. For an integer $w\ge 0$, we recall that the holomorphic $(w,0)$-differentials are precisely the kernel of $\overline{\partial}$ inside $A^{w,0}_{X}$.
	
	Ordinarily, the Dolbeault resolution of $\Omega^2_{X}$ is constructed from these sheaves, but instead, we dualize: first, write $\mathcal{A}^{p,q}_{X,c}$ for the space of \emph{compactly supported} relative $(p,q)$-differentials on $X$, carrying the same differentials as the usual ones (though they do not form a sheaf). Then we define
	\[
	\mathcal{D}_{X}^{p,q} = \hom(\mathcal{A}^{d-p,d-q}_{X,c}, \C)
	\]
	which we call the space of \emph{$(p,q)$-currents on $X$}. Unlike the compactly supported forms, these \emph{do} form a sheaf, since they can be restricted along inclusions of opens in adjunction to the pushforward of compactly supported forms. Since these pushforwards are injective, the sheaf restrictions maps are surjective and thus each $\mathcal{D}_{X}^{p,q}$ is flabby. We therefore again abusively use the same notation for the sheaf and its global sections. For these properties of currents, the original reference is \cite{dR}.
	
	The Dolbeault operators
	\[
	\partial: \mathcal{D}^{p,q}_{X}\to \mathcal{D}^{p+1,q}_{X}, \overline{\partial}: \mathcal{D}^{p,q}_{X}\to \mathcal{D}^{p,q+1}_{X}
	\]
	are defined as the graded adjoints of the exterior derivative on forms, i.e.
	\[
	(\partial c)(\eta) := (-1)^{\deg c} c(\partial \eta),(\overline{\partial} c)(\eta) := (-1)^{\deg c} c(\overline{\partial} \eta)
	\]
	where $\deg$ refers to the total degree (i.e. $p+q$ for a $(p,q)$-current).
	
	There is a natural inclusion 
	\[
	\nu_X: \mathcal{A}^{p,q}_{X} \to \mathcal{D}^{p,q}_{X/}
	\]
	given by 
	\[
	\omega\mapsto \left(\eta \mapsto \int_X \omega \wedge \eta\right)
	\]
	where $\eta$ is a compactly supported smooth $(d-p,d-q)$-form; the integral makes sense and is finite since the wedge product is also compactly supported, and of degree $(d,d)$. It is clear that $\nu_X$ commutes with the Dolbeault differentials $\partial$ and $\overline{\partial}$. Furthermore, since $(p,q)$-compactly supported forms have covariant functoriality by flat holomorphic maps and contravariant by proper holomorphic maps, by adjunction, $(p,q)$-relative currents can be pulled back by flat maps and pushed forward by proper ones. One can check that the push-pull formula for differential implies that $\nu_X$ is functorial for flat pullback and proper pushforward of complex manifolds.\footnote{For orientation-reversing maps, one has to be careful about this comparison map, since they introduce an extra sign; this was an issue we had to work with in \cite{X4}. However, all of our maps will be morphisms of complex manifolds, and hence orientation-preserving.}
	
	We now come to the main cohomological result about these currents:
	\begin{lem}
		The complex
		\begin{equation} \label{eq:currres}
			\Omega^{w}_{X} \hookrightarrow  \mathcal{D}^{w,0}_{X} \xrightarrow{\overline{\partial}} \mathcal{D}^{w,1}_{X} \xrightarrow{\overline{\partial}} \ldots
		\end{equation}
		is an acyclic resolution of sheaves. Further, the map of $\overline{\partial}$-complexes induced by $\nu_X$
		\begin{equation} \label{eq:currform}
			\mathcal{A}^{w,\bullet}_{X} \to \mathcal{D}^{w,\bullet}_{X}
		\end{equation}
		is a quasi-isomorphism for each degree $w$.
	\end{lem}
	\begin{proof}
		This is a consequence of the $\overline{\partial}$-Poincaré (or Dolbeault-Grothendieck) lemma due to \cite{Skoda} which implies the exactness of the \label{eq:currres} for sheaves of holomorphic forms/currents on any complex manifold (not relative to a base). The version for forms can be found in \cite{Serre}, which implies the quasi-isomorphism \eqref{eq:currform}.
	\end{proof}
	
	In particular, for $2$-dimensional $X$ and $w=2$, the complex 
	\[
	\mathcal{D}^{2,0}_{X} \xrightarrow{\overline{\partial}} \mathcal{D}^{2,1}_{X} \xrightarrow{\overline{\partial}} \mathcal{D}^{2,2}_{X}
	\]
	is functorial in $X$ (in the sense described previously) and computes the coherent cohomology of $\Omega^{2}_{X}$. (Here, we recall that $X$ is assumed to be relative dimension $2$, so there are no further terms.)
	
	Having proved this key cohomological property, we now introduce an important class of currents: associated to closed complex submanifolds $Z\subset X$ of codimension $r$, we have a closed \emph{current of integration}
	\[
	\delta_Z \in \mathcal{D}^{r,r}_X
	\]
	defined by  
	\[
	\delta_Z(\omega):=\int_Z \omega.
	\]
	Now consider a fixed squared elliptic fiber $T_\tau=E_\tau \times E_\tau$. Then the class of a closed current $\omega\in \mathcal{D}^{2, 1}_{T_\tau}$ having residue $\mathcal{C}\in H^0(T_\tau[c])$ along the residue map
	\[
	H^{1}(T_\tau-T_\tau[c], \Omega^2_{T_\tau - T_\tau[c]}) \to H^0(T_\tau[c])
	\]
	is equivalent to $d\omega = \delta_\mathcal{C}$ (where this is interpreted as a suitable linear combination of the currents of integration along points in the support of $\mathcal{C}$); see for example \cite[(3.3)]{X} from the author's thesis.
	
	Since the distributional Dolbeault complex is a functorial complex computing the coherent cohomology of holomorphic differentials, the coresponding \emph{equivariant} coherent cohomology of $\Omega^{2}_X$ can thus be computed by the double complex $C^\bullet(\Gamma, \mathcal{D}^{2,\bullet}_X)$, where $\Gamma\subset \GL_2(\Z)$ acts on the currents by pushforward. 
	
	Analogously to the non-equivariant case, then, an element $\omega$ of this double complex for $T$ restricts to a representative of a class in $H^{1}_\Gamma(T_\tau-T_\tau[c],\Omega^2_{T_\tau-T_\tau[c]})$ with residue 
	\[
	[T_\tau[c]-c^4\{0\}]\in H^0(T_\tau[c])^\Gamma
	\]
	if and only if the total differential of $\omega$ is $\delta_{T_\tau[c]}-c^4\delta_0 \in C^0(\Gamma, \mathcal{D}^{2,2}_{T_\tau})$. In particular, if we can find such a class which is in the \emph{trace-fixed} part $(\mathcal{D}^{2,\bullet}_{T_\tau})^{(0)}$, this notation meaning the part which is invariant by $[a]_*$ for all integers $a$ relatively prime to $c$, then it will represent the class
	\[
	Z_\tau^{(c)} \in H^{1}_\Gamma(T_\tau-T_\tau[c], \Omega^{2}_{T_\tau-T_\tau[c]}).
	\]
	
	\section{Constructing the lift} \label{section:formulas}
	\subsection{Kronecker-Eisenstein series as theta functions}
	
	We now know that in principle, one can compute Eisenstein classes by finding suitable lifts inside the distributional de Rham complex. It remains to find currents realizing these lifts.
	
	Recall the Kronecker-Eisenstein series \cite{Weil}
	\begin{equation} \label{eq:kroeis}
		K_a(s,\tau, z, u) =\sum_{\omega \in \Z+\Z \tau}' \frac{(\overline{\omega+z})^a}{|\omega+z|^{2s}} \exp\left(2\pi i\frac{\omega\overline{u}-\overline{\omega}u}{\tau - \overline{\tau}}\right)
	\end{equation}
	for $\tau \in \mathcal{H}$, $z,u\in \C$, and $a\in \Z$; the superscript apostrophe denotes that the sum omits any term where $\omega+z=0$. This series is convergent for $\text{Re}\, s>1+a/2$, but has meromorphic continuation to all $s\in \C$, with simple poles possibly only at $s=0$ (when $z\in \Z + \Z \tau$ and $a=0$) $s=1$ (when $u\in \Z + \Z \tau$ and $a=0$). We define special notations, following \cite[\S9]{BCG}, for the specializations we need:
	\begin{align}
		E_1(\tau,z) &:= \frac{i}{2\pi} K_1(1,\tau,z,0), \\
		E_2(\tau,z) & := -\frac{1}{4 \pi y} K_2(1,\tau,z,0).
	\end{align}
	Since these functions are manifestly $(\Z+\Z\tau)$-periodic in $z$ for fixed $\tau$, we can consider $E_1(\tau,z)$ and $E_2(\tau,z)$ as functions on $E_\tau$. Further, $E_1(\tau, z)$ is odd and $E_2(\tau, z)$ is even in $z$, for any $\tau$; in particular, we see that $E_1(\tau,0)=0$. As noted in \cite{Weil}, these two functions specialize at torsion points 
	\[
	E_1(\tau, \alpha-\beta\tau),E_2(\tau, \alpha-\beta\tau); (\alpha,\beta) \in (\Q/\Z)^2 \setminus \{(0,0)\}
	\]
	to holomorphic modular forms in $\tau$ of weight $1$, respectively $2$, which are the classical Eisenstein series of those weights, of level corresponding to the stabilizer of $(\alpha,\beta)$ in $\SL_2(\Z)$.
	
	From \cite[(9.6)]{BCG} together with \cite[Proposition 20]{BCG}, we observe that we have the distribution relations as \emph{functions}
	\begin{align} \label{eq:distro}
		[a]_* (E_1(\tau,z)\, dz)& = E_1(\tau,z)\, dz ,\\
		[a]_*E_2(\tau,z) & = E_2(\tau, z)
	\end{align}
	for any integer $a\in \Z\setminus\{0\}$. Further, \cite[Theorem 19]{BCG} and \cite[\S9.4]{BCG} together imply that, considered as a current on $E_\tau$, 
	\[
	\overline{\partial}E_1(\tau,z)\,dz = \delta_0 -\text{vol}_{E_\tau} = \delta_0 -\frac{2 i dz\wedge d\overline{z}}{y}
	\]
	where here we give the formula for the normalized volume form on $E_\tau$. Away from the zero section (ignoring the current of integration $\delta_0$), this derivative holds on the level of functions. Furthermore, \cite[(35)]{Weil} says that
	\[
	\overline{\partial} E_2(\tau,z) = \frac{i}{2y}\cdot E_1(\tau, z)\, d\overline{z}
	\]
	as functions (and hence currents) on $E_\tau$. Thus, both $E_1$ and $E_2$ are smooth away from the zero section, where $E_2$ is once-differentiable and $E_1$ has a singularity. Note that the formulas in \cite[\S9.4]{BCG} imply that on a punctured neighborhood of the zero section, $E_1$ looks like it has a simple log pole (even though its actual value at zero is $0$). It therefore makes sense to consider forms like $E_1(\tau, z)\,dz$ and $E_2(\tau, z)\, dz$ as $(1,0)$-currents on $E_\tau$ by considering them as kernels of integration as in the map $\nu_X$ we defined for smooth forms, since the corresponding integrals will converge for compactly supported smooth forms.
	
	\subsection{Explicit isogeny-fixed currents}
	
	We now construct a certain element of total degree $1$ (or $3$, if one adds in the Hodge weight $2$) 
	\[
	\zeta_\tau \in C^\bullet(\GL_2(\Q), (\mathcal{D}_{T_\tau}^{2,\bullet})^{(0)})
	\]
	whose total differential is the cocycle
	\[
	\delta_0-\text{vol}_{T_\tau}\in Z^0(\Gamma,(\mathcal{D}_{T_\tau}^{2,2})^{(0)})=[(\mathcal{D}_{T_\tau}^{2,\bullet})^{(0)}]^\Gamma
	\]
	
	\begin{lem}
		If we construct such a $\zeta_\tau$, then for any integer $c>1$, the restriction of $([c]^*-c^4)\zeta_\tau$ to any $\Gamma\subset \GL_2(\Z)$ and $T_\tau-T_\tau[c]\subset T_\tau$ represents the class $Z_\tau^{(c)}$.
	\end{lem}
	\begin{proof}
		If the total differential of $\zeta_\tau$ is $\delta_0-\text{vol}_E$, then the total differential of $([c]^*-c^4)\zeta_\tau$ is $\delta_{T_\tau[c]}-c^4\delta_0$. Furthermore, it remains invariant by all isogenies prime to $c$; hence by the discussion at the end of the previous section, the result follows.
	\end{proof}
	
	To construct this lift, we must specify $\zeta_\tau^{0,1}\in C^0(\GL_2(\Q),(\mathcal{D}_{T_\tau}^{2,1})^{(0)})$ and $\zeta_\tau^{1,0} \in C^1(\GL_2(\Q),(\mathcal{D}_{T_\tau}^{2,0})^{(0)})$ such that $\overline{\partial} \zeta_\tau^{1,0} = \partial_{\GL_2(\Q)} \zeta_{\tau^{0,1}}$ where the latter differential refers to the group coboundary map.
	
	We therefore fix the choice 
	\[
	\zeta_\tau^{0,1} := E_1(\tau,z_1)\,dz_1 \, \delta_{z_2=0} + E_1(\tau, z_2)\,dz_2\wedge \text{vol}_{z_1=0}
	\]
	whose image under $\overline{\partial}$ is can be computed as
	\begin{align}
		\overline{\partial} \zeta_\tau^{0,1}&= \overline{\partial} [E_1(\tau,z_2)\,dz_2 \,\delta_{z_2=0} + E_1(\tau, z_1)\,dz_1\wedge \text{vol}_{z_1=0}] \\
		&= (\delta_{0} - \text{vol}_{z_1=0} \, \delta_{z_2=0}) + ( \text{vol}_{z_1=0} \, \delta_{z_2=0} - \text{vol}_{z_1=0} \wedge \text{vol}_{z_2=0} ) \\
		&= \delta_0-\text{vol}_{T_\tau}
	\end{align}
	
	It remains therefore to find, for each $\gamma\in \GL_2(\Q)$, a lift to $(\mathcal{D}_{T_\tau}^{2,0})^{(0)}$ of
	\[
	(\gamma_*-1) \zeta_\tau^{0,1}= (\gamma_*-1)[E_1(\tau,z_2)\,dz_2 \,\delta_{z_2=0} + E_1(\tau, z_1)\,dz_1\wedge \text{vol}_{z_1=0}].
	\]
	We observe that there are no more choices to be made; these lifts are unique:
	\begin{lem}
		The trace-fixed complex
		\[
		0\to (\mathcal{D}^{2,0}_{T_\tau})^{(0)} \to (\mathcal{D}^{2,1}_{T_\tau})^{(0)} \to (\mathcal{D}^{2,2}_{T_\tau})^{(0)}
		\]
		is left-exact.
	\end{lem}
	\begin{proof}
		This can be proven identically to \cite[Lemma 6.2.1]{SV}, as the only trace-fixed cohomology of $\Omega^2_{T_\tau}$ is in top degree, coming from the form yielding the fundamental class of the $4$-torus $T_\tau$ in the Hodge-de Rham spectral sequence.\footnote{In fact, one can show using Fourier series that this complex is almost right-exact as well, except for one dimension of cohomology on the right. This is unnecessary for us, so we omit it.}
	\end{proof}
	
	To find formulas for these lifts for a fully general matrix $\gamma$ is complicated if approached directly: the pushforward action on currents can yield expressions with arbitrarily many terms. Instead, we decompose our matrices to simplify the calculation. 
	
	\subsection{Telescoping with the Bruhat decomposition} 
	
	The main observation we need is that we have the ``telescoping'' relation
	\begin{equation} \label{eq:recur}
		(\gamma_2\gamma_1-1) \zeta^{0,1}_{T_\tau} = \gamma_2(\gamma_1-1) \zeta^{0,1}_{T_\tau} + (\gamma_2-1)\zeta^{0,1}_{T_\tau}
	\end{equation}
	reducing the problem of finding a lift for the product of matrices to the problem for the individual matrices. Thus, the problem of finding lifts can be reduced to a set of generators of $\GL_2(\Q)$.
	
	We recall that the Bruhat decomposition says that 
	\[
	\GL_2(\Q) = BwB = \begin{pmatrix} * & * \\ 0 & * \end{pmatrix}\begin{pmatrix} 0 & 1 \\ 1 & 0 \end{pmatrix} \begin{pmatrix} * & * \\ 0 & * \end{pmatrix}
	\]
	where $B$ denotes the upper triangular Borel subgroup, and $w$ the antidiagonal Weyl element. We can make this more explicit by writing
	\begin{align}
		\gamma := \begin{pmatrix} a & b \\ c & d \end{pmatrix} &= \begin{pmatrix}
			1 & a \\ 0 & c
		\end{pmatrix} \begin{pmatrix}
			0& 1 \\ 1 & 0
		\end{pmatrix} \begin{pmatrix}
			1& d/c \\ 0 & - \frac{\det \gamma}{c}
		\end{pmatrix} \\ 
		&= \begin{pmatrix}
			1 & a/c \\ 0 & 1
		\end{pmatrix} \begin{pmatrix}
			1 &0 \\ 0 & c
		\end{pmatrix} \begin{pmatrix}
			0& 1 \\ 1 & 0
		\end{pmatrix} \begin{pmatrix}
			1& -\frac{d}{\det \gamma} \\ 0 & 1
		\end{pmatrix}
		\begin{pmatrix}
			1& 0 \\ 0 & -\frac{\det \gamma}{c}
		\end{pmatrix}
	\end{align}
	whenever $c\ne 0$, for an arbitrary matrix $\gamma \in \GL_2(\Q)$. Here, we have further decomposed the appearing upper triangular matrices by factoring them into a diagonal times a unipotent matrix. When $c=0$, our matrix is already in the Borel, so we can directly write
	\[
	\begin{pmatrix} a & b \\ 0 & d \end{pmatrix} = \begin{pmatrix} 1 & b/d \\ 0 & 1 \end{pmatrix} \begin{pmatrix} a & 0 \\ 0 & d \end{pmatrix} 
	\]
	
	The point of this is that $w$, diagonal, and unipotent matrices all act in very computable ways on $\zeta^{0,1}_{T_\tau}$, enabling us to find lifts:
	\begin{enumerate}
		\item All diagonal matrices act trivially on $\zeta_\tau^{0,1}$, by the isogeny properties of $E_1$ we proved earlier together with the fact that the volume form of a torus is isogeny-invariant. Hence the corresponding lifts are zero.
		\item For the element $w$, we can write $(w-1)\zeta_{T_\tau}$ as
		\[
		E_1(\tau,z_2)\,dz_2 \,\delta_{z_1=0} + E_1(\tau, z_1)\,dz_1\wedge \text{vol}_{z_1=0} - E_1(\tau,z_1)\,dz_1 \,\delta_{z_2=0} - E_1(\tau, z_2)\,dz_2\wedge \text{vol}_{z_2=0} 
		\]
		\[
		= \overline{\partial}[E_1(\tau,z_1)E_1(\tau,z_2)\,dz_1\wedge dz_2]
		\]
		which is a very simple lift.
		\item For a unipotent matrix $\gamma_u$ with upper-right entry $u$, we find that the term with the current $\delta_{z_2=0}$ cancels in $(\gamma_u-1)\zeta_{T_\tau}$ (since the coordinate $z_1$ is fixed), and what remains is
		\begin{equation} \label{eq:ubar}
			\frac{2 i u}{y} E_1(\tau,z_2)\,dz_1 \wedge dz_2  \wedge d\overline{z_2}= \overline{\partial} \left[4 u E_2(\tau,z_2) \, dz_1 \wedge dz_2\right]
		\end{equation}
	\end{enumerate}
	
	Combining these with our explicit Bruhat decomposition, we find the following formulas: in the case when $c=0$, we obtain
	\begin{equation}
		\zeta^{1,0}_{T_\tau}(\gamma) = \frac{4b}{d} E_2(\tau, z_2)\, dz_1\wedge dz_2.
	\end{equation}
	In the case when $c\ne 0$, we use the full Bruhat decomposition to obtain  the following:
	
	\begin{prop} \label{prop:lift}
		The unique lift $\zeta^{1,0}_{T_\tau}$ on a matrix $\gamma \in \GL_2(\Q)$ is given by the expression
		\begin{equation} \label{eq:gammaform}
			\left[\begin{pmatrix}a& 1 \\ c& 0 \end{pmatrix}_* \frac{4d}{c\det \gamma} E_2(\tau, z_2) + \frac{1}{c}\begin{pmatrix}1& a \\ 0& c \end{pmatrix}_* E_1(\tau, z_1)E_1(\tau, z_2) + \frac{4a}{c} E_2(\tau, z_2)\right]\, dz_1\wedge dz_2.
		\end{equation}
	\end{prop}
	
	Notice that here, we have to introduce factors coming from the determinants of the matrices to go from pushforwards of forms to pushforwards of functions.
	
	We therefore have obtained full formulas for $\zeta_{T_\tau}$, whose $c$-stabilizations yield the kernel classes $Z_\tau^{(c)}$ for the Eisenstein theta lift after restriction to $\GL_2(\Z)$.
	
	For a general matrix $\gamma$, there does not appear to be a substantial further simplification of the preceding formulas, besides writing out the pushforwards as sums.
	
	\subsection{Specializations at torsion points}
	
	The actual ``theta lift'' for $(\GL_2,\GL_2)$ is generally taken to be valued after contraction pullback by torsion sections. Fix an arithmetic subgroup $\Gamma\subset \GL_2(\Z)$; in order to compare our results with the cohomological theta lift, we will restrict $\zeta_\tau$ to $\Gamma$.
	
	\begin{rem}
		This arithmeticity hypothesis is only necessary to compare with the cohomological construction of the theta lift: one can still obtain pulled-back cocycles on, say, $S$-arithmetic subgroups for some set of inverted places $S$, so long as they fix some torsion sections (of order necessarily prime to the places in $S$). We do not write down this extension here, but the formulas can be obtained by our same methods.
	\end{rem}
	
	In this context, we wish to consider the image of $Z_\tau^{(c)}$ under the composite
	\begin{equation} \label{eq:comp}
		H^1_\Gamma(T_\tau-T_\tau[c], \Omega^2_{T_\tau-T_\tau[c]}) \xrightarrow{\cap \frac{\partial}{\partial z_1}\otimes \frac{\partial}{\partial z_2}} H^1_\Gamma(T_\tau-T_\tau[c], \mathcal{O}_{T_\tau-T_\tau[c]}) \xrightarrow{D^*} H^1_\Gamma(\{\tau\}, \C) = H^1(\Gamma, \C)
	\end{equation}
	where $D$ is any $\Gamma$-fixed torsion cycle $D$ disjoint from $T_\tau[c]$. (Here, we are slightly abusive in writing the pullback $D^*$; this is actually a sum of pullbacks over the various torsion sections in the support of $D$.) Analogous to the case of a single torsion section, we will write $\Theta_{\tau,D}^{(c)}$ for this image.
	
	We wish to interpret this composite in terms of the explicit double complex representative $\zeta_\tau$; the main issue is that currents cannot in general be pulled back by closed immersions.
	
	The technical tool we need to remedy this is the introduction of a variant of the distributional Dolbeault complex: For any $\Gamma\subset \GL_2(\Q)$, let $H_\Gamma\subset T_\tau$ be the $\Gamma$-orbit of the lines $\{z_1=0\}$ and $\{z_2=0\}$ and their translates by $c$-torsion sections inside $T_\tau$; this is a union of infinitely many elliptic subschemes (and their $c$-torsion translates). For any finite subarrangement $H\subset H_\Gamma$, we define a complex $\mathcal{D}_{T_\tau,H}^{2,\bullet}$ via the pullback square 
	\begin{equation}
		\begin{tikzcd}
			\mathcal{D}_{T_\tau,H}^\bullet \arrow[r] \arrow[d] & \mathcal{D}_{T_\tau}^{2,\bullet} \arrow[d]\\
			\mathcal{A}_{T_\tau-H}^{2,\bullet} \arrow[r] & \mathcal{D}_{T_\tau-H}^{2,\bullet}
		\end{tikzcd}
	\end{equation}
	which results in an identification of $\mathcal{D}_{T_\tau,H}^{2,i}$ with the $(2,i)$-currents such that their restriction to $T_\tau-H$ are given by smooth $(2,i)$-forms. Here, the bottom horizontal map is the earlier-defined inclusion, and the right vertical map is the restriction dual to the pushforward of compactly-supported differential forms. We define then
	\[
	\mathcal{D}_{T_\tau,H_\Gamma}^{2,\bullet}:= \varinjlim_H \mathcal{D}_{T_\tau,H}^{2,\bullet} 
	\]
	where the limit runs over $H$ finite subarrangements of $H_\Gamma$, along the natural inclusion maps. The group $\Gamma$ permutes the pullback diagrams for each $H$ (sending it to that of $\gamma H$), and these assemble to give a pushforward action on $\mathcal{D}_{T_\tau,H_\Gamma}^{2,\bullet}$. Further, because the bottom row in each pullback diagram is a quasi-isomorphism, we see that $\mathcal{D}_{T_\tau,H_\Gamma}^{2,\bullet}$ computes the cohomology of $\Omega^2_{T_\tau}$, just as $\mathcal{D}_{T_\tau}^{2,\bullet}$ does. Furthermore, analogously to the full distributional de Rham complex, we have a left exact sequence
	\[
	0 \to (\mathcal{D}_{T_\tau,H_\Gamma}^{2,0})^{(0)}\to (\mathcal{D}_{T_\tau,H_\Gamma}^{2,1})^{(0)}\to (\mathcal{D}_{T_\tau,H_\Gamma}^{2,2})^{(0)}
	\]
	meaning that $\zeta_\tau$, considered in this more refined complex, is still uniquely determined. The important new phenomenon for us is that if $\Gamma$ fixes the torsion cycle $D\subset T_\tau$ disjoint from $H_\Gamma$, then there is a composite pullback map
	\[
	\mathcal{D}_{T_\tau,H_\Gamma}^{2,0}\to \varinjlim_H \Omega^2_{T_\tau-H} \xrightarrow{\cap \frac{\partial}{\partial z_1}\otimes \frac{\partial}{\partial z_2}} \mathcal{O}_{T_\tau-H} \xrightarrow{D^*} \C
	\]
	which induces the composite \eqref{eq:comp}.
	
	We now concern ourselves with the case that $D$ is supported on $N$-torsion for an integer $N>1$. From the preceding discussion, we can conclude that 
	\[
	[D^* ([c]^*-c^4) \zeta_\tau] = \Theta_{\tau, D}^{(c)}
	\]
	so long as $D$ is disjoint from $H_\Gamma$. Note that if $c\equiv 1\pmod{N}$, we can write the left-hand side as 
	\[(1-c^4)[D^*\zeta_\tau].\]
	
	The restriction on $D$ is rather irritating, as it depends on the arbitrary choice of coordinates $z_1,z_2$ we used to choose our lift $\zeta_\tau^{0,1}$: for any given torsion section, we could simply start with a different lift to obtain formulas for the pullback. However, this would result in a somewhat unsatisfying lack of unity in our formulas. Luckily, we can use a trick to bypass this issue entirely, and make the formulas valid even for ``bad'' torsion sections:
	
	\begin{thm} \label{thm:inj}
		Suppose $c\equiv 1\pmod{N}$. Then for \emph{any} $N$-torsion cycle $D$ disjoint from the identity, we have that
		\[
		[D^* \zeta_\tau] = \frac{1}{1-c^4}\Theta_{\tau, D}^{(c)}
		\]
	\end{thm}
	The proof consists of ``bootstrapping'' from torsion points disjoint from $H_\Gamma$ to all of them. In order to do this, we will need the following lemma allowing us to ``improve'' our current-valued cocycles to be form-valued:
	
	\begin{lem} \label{lem:inj}
		There is an injection
		\[
		M_\tau \hookrightarrow (\mathcal{D}^{2,0}_{T_\tau,H_\Gamma})^{(0)}
		\]
		where $M_\tau$ is defined to be the module of $(2,0)$-forms on $T_\tau-\{0\}$ spanned by the $\GL_2(\Q)$-orbit of $E_2(z_2)\,dz_2$ and $E_1(\tau, z_1)E_1(\tau, z_2)\, dz_1\wedge dz_2$, given by sending
		\[
		\omega \mapsto \left(\eta\mapsto \int_T \omega \wedge \eta\right).
		\]
	\end{lem}
	\begin{proof}[Proof of lemma]
		The only non-formal assertion here is that this map is injective. Notice that the map considering \emph{smooth} forms (i.e. the map $\nu_{T_\tau}$ defined before) as currents via kernels of integration, or even continuous forms, is clearly injective: by integrating against $d\overline{z}_1\wedge d\overline{z}_2$ times a bump function on any small open set, we see that the zero current can only come from a form which vanishes almost everywhere, which hence must be zero by continuity. Thus, the depth of this lemma's assertion comes precisely from the discontinuities of the forms in the orbit of $E_1(\tau,z_1)E_1(\tau,z_2)\,dz_1\wedge dz_2$ along codimension-$1$ sub-elliptic curves.\footnote{Indeed, to appreciate the delicacy, observe if we change $T_\tau-\{0\}$ to $T$ in its statement, the statement becomes false: see \cite[Proposition 3.7]{BG} for an example of a relation between the weight-$1$ and weight-$2$ series everywhere except the zero section.}
		
		We note the following property of $ E_1(\tau,z_1)E_1(\tau,z_2)\,dz_1\wedge dz_2$: suppose $x=(x_1,x_2)\in T_\tau$ is a point lying on one of the subcurves of discontinuity $S\subset T_\tau$ (so $x_1=0$ or $x_2=0$) but not equal to zero. Then take any small $v =(v_1,v_2)\in \C^2$ not parallel to the curve of discontinuity of $x$, so that $x\pm v$ does not lie in $S$. By the oddness of $E_1$, we find that the average of the translates by $\pm v$ vanishes as we shrink $v$: 
		\[
		\lim_{\epsilon\to 0^+} \frac{1}{2}(E_1(\tau,z_1+\epsilon v_1)E_1(\tau,z_2+\epsilon v_2)+E_1(\tau,z_1-\epsilon v_1)E_1(\tau,z_2-\epsilon v_2)) = 0
		\]
		as an equality of coefficients of $dz_1\wedge dz_2$.
		
		By moving this argument around by the general linear action, this applies to \emph{any} nonzero point on a codimension $1$ discontinuity stratum of a function in the orbit of $E_1(\tau, z_1)E_1(\tau, z_2)\,dz_1\wedge dz_2$.
		
		Now consider an arbitrary $\omega\in M_\tau$, and suppose that $\omega$ gives the trivial $(2,0)$-current when considered as a kernel of integration. Then, $\omega$ must be identically zero outside a finite union of sub-elliptic curves through the identity. Consider an arbitrary nonzero point $x$ on one of these subcurves $S\subset T_\tau$, and pick some decomposition
		\[
		\omega= \omega_1+\omega_2
		\]
		where $\omega_1$ consists of a sum of terms in the orbit of $E_2(\tau, z_2)\,dz_1\wedge dz_2$ or $E_1(\tau, z_1)E_1(\tau, z_2) \,dz_1\wedge dz_2$ which \emph{do not} have a discontinuity along $S$, and $\omega_2$ consists of a sum of terms from the latter orbit which \emph{do} have a discontinuity along $S$. Pick now some vector $v\in C$ such that the the line segment between the points $x\pm v$ intersects no other discontinuity locus of any term in $\omega_1$ or $\omega_2$; this is always possible since $v$ is nonzero and there are only finitely many terms to consider (and hence subcurves to avoid).
		
		Then we find that
		\begin{align}
			\frac{1}{2}\left(\lim_{\epsilon\to 0^+} \omega_2|_{x+\epsilon v} + \omega_2|_{x-\epsilon v}\right) &= \frac{1}{2}\left(\lim_{\epsilon\to 0^+} \omega|_{x+\epsilon v}-[\omega_1]|_{x+\epsilon v} + \omega|_{x-\epsilon v}-[\omega_1]|_{x-\epsilon v}\right) \\
			&= \frac{1}{2}\left(\lim_{\epsilon\to 0^+}-[\omega_1]|_{x+\epsilon v}-[\omega_1]|_{x-\epsilon v}\right)
		\end{align}
		since $\omega$ is identically zero on a neighborhood of $x$ in $T_\tau-S$. But $\omega_1$ is continuous in a neighborhood of $x$ by assumption, so this expression is just $-\omega_1|_{z=x}$. On the other hand, the average of the two limits we started with is zero from the preceding discussion, so we conclude that $\omega_1|_{z=x}=0$. We also have $\omega_2|_{z=x}=0$ because forms in the orbit of $E_1(\tau,z_1)E_1(\tau,z_2)\, dz_1\wedge dz_2$ are zero along their discontinuity loci by construction. We hence conclude that $\omega|_{z=x}=0$; since this applies to any nonzero point $x$, we conclude that $\omega$ is the zero form on $T_\tau-\{0\}$. This concludes the proof of injectivity.
	\end{proof}
	
	\begin{proof}[Proof of theorem]
		Thanks to the lemma, we can consider $\zeta_\tau$ to be a cocycle valued in $M_\tau$. Let $T_\tau[N]'$ denote the primitive $N$-torsion, and let $S \subset T_\tau$ be a sub-elliptic curve which is not the vanishing locus of either $z_1$ or $z_2$. Let $x_0\in S[N]'$ be any point, and write $\Gamma_1(x_0)\subset \GL_2(\Z)$ for its stabilizer. By construction, $x\not\in H_{\Gamma_1(x_0)}$, and so by the previous discussion, 
		\begin{equation} \label{eq:strap}
			[x_0^*\zeta_\tau] = \frac{1}{1-c^4}\Theta^{(c)}_{\tau, x_0}
		\end{equation}
		for any $c\equiv 1\pmod{N}$. In fact, noticing that $\Gamma_1(x_0)$ must stabilize the entire elliptic sub-curve $S$, this formula holds for any point $x\in S[N]'$.
		
		We now observe that $z_{\tau}^{(c)} \cap \frac{\partial}{\partial z_1} \otimes \frac{\partial}{\partial z_2}$ induces, by restriction, a class
		\begin{equation} \label{eq:boot}
			\Theta_{\tau,N}^{(c)}\in H^1_{\GL_2(\Z)}(T_\tau[N]',  \C) \cong H^1(\GL_2(\Z), \hom(T_\tau[N]', \C))
		\end{equation}
		where the isomorphism comes from the fact that $T_\tau[N]'$ is a union of contractible spaces, causing the collapse of the Hochschild-Serre spectral sequence. By functoriality of this spectral sequence, for any cycle $D\subset T_\tau[N]'$ stabilized by $\Gamma(D)$, the image of $\Theta_{\tau,N}^{(c)}$ under the composite of restriction and evaluation
		\[
		H^1(\GL_2(\Z), \hom(T_\tau[N]', \C))\xrightarrow{\text{res}} H^1(\Gamma(D), \hom(T_\tau[N]', \C)) \xrightarrow{D} H^1(\Gamma(D), \C)
		\]
		yields $\Theta^{(c)}_{\tau, D}$. 
		
		Observe that there is a $\GL_2(\Z)$-equivariant map
		\[
		M_\tau\to \hom(T_\tau[N]', \C), f\, dz_1\wedge dz_2 \mapsto \left(x\mapsto f(x)\right).
		\]
		We claim that the pushforward of $([c]^*-c^4)\zeta_\tau$ under this map can be identified with $\Theta_{\tau,N}^{(c)}$, which would then imply the desired result for arbitrary primitive $N$-torsion cycles.
		
		Indeed, there is an isomorphism of $\GL_2(\Z)$-modules
		\[
		\hom(T_\tau[N]', \C) \to \text{Ind}_{\Gamma_1(x_0)}^{\GL_2(\Z)}\, \hom(S[N]', \C), f\mapsto \left(\gamma\mapsto f\circ \gamma^{-1}\right)
		\]
		where $\hom(S[N]', \C)$ has, naturally, a trivial action of $\GL_2(\Z)$. Hence, by Shapiro's lemma it suffices to show that $([c]^*-c^4)\zeta_\tau$ and $\Theta_{\tau,N}^{(c)}$ agree upon restriction to $\Gamma_1(x_0)$ under the quotient
		\[
		\hom(T_\tau[N]', \C) \twoheadrightarrow \hom(S[N]', \C) 
		\]
		dual to the obvious inclusion. But this is precisely \eqref{eq:strap}, which we have already established. Assembling these identifications together for all $N>1$ yields the full theorem.
	\end{proof}
	
	Thus, from \eqref{eq:gammaform}, when specialized at any nonzero torsion sections (or combination thereof), yields the following formula for the Eisenstein theta lift of \cite{BCG}:
	
	\begin{thm} \label{thm:qthm}
		Let 
		\begin{equation} \label{eq:formula}
			\theta_\tau[\gamma](z_1,z_2):= \begin{cases}
				\frac{4b}{d} E_2(\tau, z_2) &\text{if $c=0$, else} \\
				
				\begin{pmatrix}a& 1 \\ c& 0 \end{pmatrix}_* \frac{4d}{c \det \gamma} E_2(\tau, z_2) + \frac{1}{c}\begin{pmatrix}1& a \\ 0& c \end{pmatrix}_* E_1(\tau, z_1)E_1(\tau, z_2) + \frac{4a}{c} E_2(\tau, z_2) 
			\end{cases}
		\end{equation}
		Then given any $\Gamma$-fixed combination of nonzero torsion sections 
		\[D = \sum_i t_i [(u_i, v_i)],\]
		we have
		\[
		\Theta^{(c)}_{\tau, D}(\gamma) = \sum_i c_i \theta_\tau[\gamma](a_1,b_2) \in \C.
		\]
		Here, $t_i$ are integer coefficients, and $u_i$ and $v_i$ are elements of $(\Z/N)^2-\{(0,0)\}$, thought of as $N$-torsion sections on $E_\tau$.
	\end{thm}
	Noticing that $\theta_\tau$ transforms like a weight-$2$ modular form in $\tau$, we can consider its specialization at torsion sections as a section of the weight-$2$ automorphic line bundle $\omega^2$ on any open modular curve over which the torsion sections are defined. It immediately follows:
	\begin{cor} \label{cor:qcor}
		If $H$ is any level structure fixing the torsion cycle $D$, then with the same notation as above, we have
		\[
		\Theta^{(c)}_{\Gamma, D}(\gamma) = \sum_i c_i \theta_\tau[\gamma](u_i,v_i) \in H^0(Y(H), \omega^2).
		\]
		where $Y(H)$ is the open modular curve of level $H$.\footnote{It is also true that $\theta_\tau$ represents cocycles over distributions of torsion sections, where $\Gamma$ now acts nontrivially by permuting the sections. Because of the way we set up our machinery in this article, this is not immediate; however, it can be proven with only a little extra work.}
	\end{cor}
	
	These formulas are workable, but the presence of the pushforward matrices (which can be evaluated as finite sums over preimage torsion sections) make them slightly unwieldy. Analogously to the classical setting of periods of Eisenstein series \cite{X4}, the first term 
	\[
	\sum_i c_i \theta_\tau[\gamma](u_i, v_i) = \frac{4d}{c \det \gamma}\sum_i c_i \left[\begin{pmatrix}a& 1 \\ c& 0 \end{pmatrix}_*  E_2(\tau, z_2)\right]_{(u_i,v_i)}
	\]
	can be simplified if we assume that $\gamma\in SL_2(\Z)$: 
	\begin{align}
		\frac{4d}{c\det \gamma}\sum_i c_i \left[\begin{pmatrix}a& 1 \\ c& 0 \end{pmatrix}_*  E_2(\tau, z_2)\right]_{(u_i,v_i)} &= 4\frac{d}{c}\sum_i c_i \sum_{j,k \in (\Z/c)^2} E_2\left(\tau, u_i - \frac{a}{c} (v_i+(j,k))\right) \\
		& = 4\frac{d}{c}\sum_i c_i \sum_{j,k \in (\Z/c)^2} E_2\left(\tau, \frac{1}{c}(v_i+(j,k))\right) \\
		& = \frac{4d}{c} \sum_i c_i E_2(\tau, v_i) \\ 
	\end{align}
	Here, we use the distribution property of $E_2$, along with the fact that $\gamma^{-1}$ stabilizes $D$. This latter fact implies that for all $i\in I$,
	\[
	av_i-cu_i = v_{\sigma(i)}
	\]
	for some permutation $\sigma$ of the index set $I$ such that $c_i=c_{\sigma(i)}$ for all $i\in I$.
	
	Unfortunately, we do not see a natural way to simplify the $E_1E_1$ term in any generality, analogously to the classical formulas for Eisenstein periods we discussed in \cite{X4}.
	
	\begin{rem}
		Note that the value at $\tau=\infty$ theta lift $\theta_\tau$ yields precisely the classical formula for the weight-$2$ Eisenstein cocycle reproven in loc. cit; this follows immediately from the fact that at $\tau=\infty$, the series $E_i(\tau,z)$ degenerates to the periodic Bernoulli polynomial $\frac{B_i(z)}{2i}$: this is immediate from the description of both functions by Hecke regularized (analytic continuation in $s$), since it holds for $s$ with large enough real part that the Fourier series are absolutely convergent.
		
		Hence as expected, the $(\GL_2, \GL_2)$-Eisenstein theta lift's degeneration at a cusp yields the $(\GL_2,\GL_1)$ theta lift (in the sense described in \cite[\S13]{BCG}).
	\end{rem}
	
	\begin{rem}
		Instead of a formula, one can obtain a more efficient "continued fraction" algorithm for computing the lifts for matrices in $\SL_2(\Z)\subset \GL_2(\Q)$ by using its famous generators $S$ and $T$ and the recursion principle \eqref{eq:recur}: this is presented for Eisenstein cocycles presented in \cite[\S2.4]{Scz1}, but works identically here by replacing the Bernoulli polynomials with our Eisenstein-Kronecker series. The outputs of this algorithm will coincide with the preceding formulas by uniqueness of $\zeta^{1,0}_{T_\tau}$, though this is not visibly obvious.\end{rem}
	
	\section{Some properties and applications}
	
	\subsection{Hecke equivariance}
	
	As one expects for a theta lift, the cocycles $\theta_\tau$, considered for the group $\SL_2(\Z)$, satisfy a compatbility property between two kinds of Hecke operators: geometric Hecke operators coming from the variable $\tau$ in the upper half-plane with its $\SL_2(\Z)$-action, and a cohomological Hecke action coming from the matrix action fiberwiise. 
	
	Using our algebraic approach, one could prove this compatibility analogously to the approach in \cite[\S6]{SV}. However, since a form of Hecke compatibility was already proven in \cite[Théorème 2.8]{BCGV} for a closely related cocycle, it is much easier for us to simply to import this result using our already-proven comparison.
	
	To fix ideas, in this section we will consider the restriction of $\theta_{\tau}$ to $\Gamma:=\Gamma_1(N)\subset \SL_2(\Z)$ for some integer $N>1$, and a torsion section $x=(0, x_2): \mathcal{H} \to E^2$ which descends to level $Y_1(N)$. The below approach can be applied to broader contexts, but in this article we will remain in this setting.
	
	We recall the definition of two kinds of Hecke operators for $\GL_2$ acting on $[\theta_{\tau}]$: a fiberwise action coming from group cohomology, and a geometric action coming from the Möbius action on $\tau$. 
	
	We write $\Delta$ to be the monoid of rank-$2$ integral matrices which stabilize $(1,0)\in (\Z/N)^2$ for the standard left representation, so that $\Gamma\subset \Delta$. Given any double $\Gamma$-coset in $\Delta$, we can decompose it finitely as
	\[ 
	\Gamma \alpha \Gamma = \bigcup_i \alpha_i \Gamma.
	\]
	As always, there are two different $\Delta$-actions we need to consider: first, the ``fiberwise'' action, where $\gamma$ acts by
	\[
	(|\det \gamma|\gamma^{-1})^*,
	\]
	this choice made so that for $\gamma\in \SL_2(\Z)$ it coincides with the pushforward we have heretofore been considering, and the ``modular'' action, which sends
	\[
	\gamma\cdot (\tau, z_1,z_2):= \left(\frac{a\tau + b}{c\tau+d},\frac{z_1}{c\tau + d}, \frac{z_2}{c\tau+d}\right).
	\]
	Then for any double coset $\Gamma \alpha \Gamma$, the action of $\Gamma\alpha \Gamma$ on $1$-cocyles can be defined as in \cite{RW} (or \cite[\S2.2.1]{BCGV}) by sending a $1$-cocycle $c:\Gamma\to M$ valued in a $\Delta$-module $M$ to
	\[
	\gamma \mapsto \sum_i \alpha_i c(\gamma_i)
	\]
	where $\gamma_i$ is defined by the relation $\alpha_i \gamma = \gamma_i \alpha_{\sigma(i)}$ for some permutation $\sigma$ of the representatives $\alpha_i$. (Note that our conventions differ slightly from loc. cit, both here and for the pullback action of $\Delta$; these two changes result in the same Hecke action.) 
	
	On invariants (with the modular or fiberwise action), the action of $\Gamma \alpha \Gamma$ is simpler to define, sending an element $x\in M^\Gamma$ to
	\[
	\sum_i \alpha_i x.
	\]
	If we denote the fiberwise Hecke operator by $T(\alpha)$ and the modular one by $\mathbf{T}(\alpha)$, then (4), (5) of \cite[Théorème 2.8]{BCGV} tell us that
	\[
	T(\alpha)\circ \theta_{\tau, D} = T(\alpha)\circ \theta_{\tau, T(\alpha) D}, \mathbf{T}(\alpha)\circ \theta_{\tau, D} = T(\alpha)\circ \theta_{\tau, \mathbf{T}(\alpha) D} 
	\]
	for any $\alpha \in \Delta$. In particular, let $T_p$ and $\mathbf{T}_p$ be the double coset operators associated to a prime $p$, consisting of all matrices in $\Delta$ with determinant $p$. When $p$ is relatively prime to $N$, these form just a single double coset, else they may be a sum of multiple such operators.
	
	Write $\delta_p$ for the torsion cycle comprised of all $p$-torsion points $((\alpha_1,\beta_1),(\alpha_2,\beta_2))\in T_\tau[p]$ such that $(\alpha_1,\beta_1)$ and $(\alpha_2,\beta_2)$ are linearly dependent over $\Z/p$. Then we can compute that for any auxiliary integer $c>1$, we have
	\[
	T_p(T_\tau[c]-c^4\{0\})=
	\mathbf{T}_p(T_\tau[c]-c^4\{0\})=([c]^*-c^4)(\delta_p+p\{0\}).
	\]
	Pulling back this equality by $x$ (since this commutes with the pullback action of $\Delta$), we obtain:
	
	\begin{prop}
		We have the equality $T_p\theta_{\tau,x}=\mathbf{T}_p\theta_{\tau, x}$ for all primes $p$, i.e. 
		\[
		\theta_{\tau,x}:H_1(\Gamma_1(N),\Z) \to H^0(Y_1(N),\omega^{\otimes 2})
		\]
		is equivariant for the Hecke subalgebra generated by $\{T_p\}_p$ for the fiberwise, respectively modular Hecke actions on source and target.
	\end{prop} 
	
	Note in particular that this includes the entire anemic Hecke algebra (all operators of level prime to $N$), but does not necessarily include $U_p$ for $p$ dividing $N$.
	
	Using this Hecke equivariance, one can obtain explicit spectral decompositions of the $(\GL_2,\GL_2)$-theta lift, by using the Rankin-Selberg formula for the inner product of an eigenform with the $E_1E_1$ terms in \eqref{eq:formula}. However, we will pursue this via a more systematic approach in future work, so do not go into it here.
	
	\subsection{CM elliptic curves}
	
	Suppose now that $\tau$ satisfies a quadratic equation with rational coefficients; then the corresponding elliptic curve $E_\tau$ has complex multiplication by an order $\mathcal{O}$ in $K:=\Q(\tau)$. Then we can extend the action of $\GL_2(\Z)$ on $T_\tau$ to an action of $\GL_2(\mathcal{O})$, and therefore the action of $\GL_2(\Q)$ on the trace-fixed distributional de Rham complex to an action of $\GL_2(K)$. Note that in this case, we can take ``trace-fixed'' to include all isogenies built out of the ``scalar'' endomorphisms in $\mathcal{O}$, because \eqref{eq:distro} generalizes to these isogenies \cite[Proposition 1.1.6]{BK}.
	
	Since this latter group has a Bruhat decomposition 
	\[
	\GL_2(K) = B_K\begin{pmatrix}&1\\1& \end{pmatrix}B_K
	\]
	(where here $B_K$ denotes the upper triangular Borel of this group) exactly as over $\Q$, the arguments of section \ref{section:formulas} go through exactly as before, with the small detail that in equation \ref{eq:ubar}, the lift must be 
	\[
	4\overline{u} E_2(\tau, z_2)\, dz_1\wedge dz_2.
	\]
	We hence conclude that the corresponding map given by
	\[
	\gamma\mapsto \theta_\tau[\gamma](z_1,z_2):=     \begin{cases}
		\frac{4b}{d} E_2(\tau, z_2) &\text{if $c=0$, else} \\
		
		\begin{pmatrix}a& 1 \\ c& 0 \end{pmatrix}_* \frac{4\overline{d}}{\overline{c \det \gamma}} E_2(\tau, z_2) + \frac{1}{c}\begin{pmatrix}1& a \\ 0& c \end{pmatrix}_* E_1(\tau, z_1)E_1(\tau, z_2) + \frac{4\overline{a}}{\overline{c}} E_2(\tau, z_2) 
	\end{cases}
	\]
	is a cocycle for $\GL_2(K)$ valued in functions on $T_\tau - \{0\}$, whose restriction to $\GL_2(\mathcal{O})$ comes from equivariant polylogarithm class. Just as with the $\GL_2(\Q)$ cocycle, this can be specialized at various torsion points. In this case, these specializations are just numbers instead of varying over an underlying symmetric space, so the ``big cocycle'' valued in forms on $T_\tau$ may be of primary interest. 
	
	This imaginary quadratic cocycle can be viewed as being ``for the dual pair $(\GL_2(K),\GL_1(K))$,'' and is approached analytically in the work \cite{BCGIm}. In particular, our formula above gives a simple expression in terms of Eisenstein-Kronecker numbers of the weight-$(0,0)$ cocycle denoted $\Phi^{0,0}$ in loc. cit., and therefore also the values of Hecke $L$-functions associated to weight-$0$ characters for the field $K$, as in \cite[Theorem 1.2]{BCGIm}. As in the $\GL_2(\Q)$ case, by employing twisted versions of our complexes and taking different weight Eisenstein-Kronecker series, it is possible to obtain analogous formulas for the more general cocycles $\Phi^{p,q}$, which could be an interesting direction of future work.
	
	\printbibliography

@article{BCG,
	author={Nicolas Bergeron and Pierre Charollois and Luis Garcia},
	title={Transgressions of the Euler class and Eisenstein cohomology of $\text{GL}_N(\mathbb{Z})$.},
	eventtitle={2018 Takagi lecture notes},
	year = 2020,
	shorthand="BCG1",
	pages = {311-379},
	journal={Japanese Journal of Mathematics},
	volume=15
}

@preprint{BCGIm,
	author={Nicolas Bergeron and Pierre Charollois and Luis Garcia},
	title={Eisenstein cohomology classes for $\text{GL}_N$ over imaginary quadratic fields.},
	year = 2021,
	shorthand="BCG2",
	url="https://drive.google.com/file/d/18_36dIOfsmoCmd3Q-6GT4OgPrTYCBYlL/view?usp=sharing"
}

@preprint{BCGV,
   author={Nicolas Bergeron and Pierre Charollois and Luis Garcia},
   title={Cocycles de Sczech et arrangements d’hyperplans dans les cas additif, multiplicatif et elliptique.},
   shorthand="BCG3",
   note="In preparation."
}

@article{BG,
	author={Lev A. Borisov and Paul E. Gunnells},
	title={Toric modular forms and nonvanishing of $L$-functions.},
	journal={Journal für die reine und angewandte Mathematik},
	volume = 539,
	year = 2001,
	pages = {149-165},
	shorthand="BG"
}

@preprint{BPPS,
	author={Cecilia Busuioc and Jeehoon Park and Owen Patashnick and Glenn Stevens},
	title={Modular symbols with values in Beilinson-Kato distributions.},
	year=2023,
	shorthand="BPPS",
    url="https://arxiv.org/abs/2311.14620"
}

@article{BK,
	author={Kenichi Bannai, Shinichi Kobayashi},
	title={Algebraic theta functions and $p$-adic interpolation of Eisenstein-Kronecker numbers.},
	journal={Duke Math Journal.},
    volume=153,
    number=2,
    year=2010,
    pages={229-295},
    shorthand="BK"
}

@phdthesis{D,
	author       = {Clément Dupont}, 
	title        = {Périodes des arrangements d’hyperplans et coproduit motivique.},
	school       = {Université Pierre et Marie Curie},
	year         = 2014,
	month        = 9,
	shorthand = "Dup"
}

@article{GL,
	author={Thomas Geisser and Marc Levine},
	title={The $K$-theory of fields in characteristic $p$.},
	journal={Inventiones mathematicae.},
	volume = 139,
	year = 2000,
	pages = {459-493},
	shorthand="GL"
}

@article{H,
	author={Haruzo Hida},
	title={Galois representations into $GL_2(\mathbb{Z}_p[[X]])$ attached to ordinary cusp forms.},
	journal={Inventiones},
	volume=85,
	pages={545-613},
	year=1986,
	shorthand="H"
}

@preprint{KS,
	author={Guido Kings and Johannes Sprang},
	title={Eisenstein-Kronecker classes, integrality of critical values of Hecke $L$-functions and $p$-adic interpolation.},
	note="Preprint.",
	url="https://arxiv.org/pdf/1912.03657.pdf"
}

@article{O,
	author={Tadashi Ochiai},
	title={A generalization of the Coleman map for Hida deformations.},
	journal={American Journal of Mathematics.},
	volume=125,
	number=4,
	pages={849-892},
	year=2003,
	shorthand="O"
}

@book{dR,
	author={Georges de Rham},
	title={Variétés différentiables: formes, courants, formes harmoniques.},
    publisher={Hermann},
	year=1955,
	shorthand="dR"
}

@article{RW,
	author={Young Rhie and George Whaples},
	title={Hecke operators in the cohomology of groups.},
	journal={Journal of the Mathematical Society of Japan.},
	volume=22,
	pages={431-442},
	year=1970,
	shorthand="RW"
}

@article{Serre,
	author={Jean-Pierre Serre},
	title={Faisceaux analytiques sur l'espace projectif.},
	journal={Séminaire Henri Cartan.},
	volume=6,
	pages={1-10},
	year=1953,
    number=18,
	shorthand="Serre"
}

@article{Skoda,
	author={Henri Skoda},
	title={A Dolbeault lemma for temperate currents.},
	journal={Annales Henri Lebesgue.},
	volume=4,
	pages={879-896},
	year=2021,
	shorthand="Skoda"
}

@article{Scz1,
	author={Robert Sczech},
	title={Eisenstein cocycles for $\GL_2(\Q)$ and values of $L$-functions in real quadratic fields.},
	journal={Commentarii Mathematici Helvetici.},
	volume=67,
	pages={363-382},
	year=1992,
	shorthand="Scz1"
}

@preprint{SV,
	author={Romyar Sharifi and Akshay Venkatesh},
	title={Eisenstein cocycles in motivic cohomology.},
	note="Preprint.",
	year=2020,
	shorthand="SV",
	url="https://arxiv.org/pdf/2011.07241"
}

@book{Weil,
	author={André Weil},
	title={Elliptic Functions According to Eisenstein and Kronecker.},
	year=1999,
	publisher={Springer-Verlag},
	series={Classics in Mathematics},
	shorthand="Weil",
    notes="Reprint."
}

@involume{W,
	author={Andrzej Weber},
	title={Leray spectral sequence for complements of certain arrangements of smooth submanifolds.},
	booktitle={Configuration Spaces.},
	publisher={Springer},
	year = 2016,
	volume=14,
	pages = {107–118},
	shorthand="W"
}

@phdthesis{X,
	author       = {Peter Xu}, 
	title        = {Arithmetic Eisenstein theta lifts.},
	school       = {McGill University},
	year         = 2023,
	shorthand = "X"
}

@preprint{X3,
	author={Peter Xu},
	title={Explicit arithmetic Eisenstein cocycles: the elliptic case.},
	year=2024,
	shorthand="X3",
	note={In preparation.}
}

@preprint{X4,
	author={Peter Xu},
	title={Classical periods of Eisenstein series and Bernoulli polynomials in the equivariant cohomology of a torus.},
	year=2024,
	shorthand="X4",
	note={Preprint.}
}
	
\end{document}